\documentclass[oneside,english]{amsart}
\usepackage[T1]{fontenc}
\usepackage[latin9]{inputenc}
\usepackage{units}
\usepackage{amstext}
\usepackage{amsthm}
\usepackage{amssymb}
\usepackage{esint}

\makeatletter
\numberwithin{equation}{section}
\numberwithin{figure}{section}
  \theoremstyle{plain}
  \newtheorem*{thm*}{\protect\theoremname}
  \theoremstyle{remark}
  \newtheorem*{rem*}{\protect\remarkname}
\theoremstyle{plain}
\newtheorem{thm}{\protect\theoremname}
 \theoremstyle{definition}
 \newtheorem*{defn*}{\protect\definitionname}
  \theoremstyle{plain}
  \newtheorem{lem}[thm]{\protect\lemmaname}
  \theoremstyle{plain}
  \newtheorem{cor}[thm]{\protect\corollaryname}

\usepackage{dsfont}

\makeatother

\usepackage{babel}
  \providecommand{\corollaryname}{Corollary}
  \providecommand{\definitionname}{Definition}
  \providecommand{\lemmaname}{Lemma}
  \providecommand{\remarkname}{Remark}
  \providecommand{\theoremname}{Theorem}
\providecommand{\theoremname}{Theorem}

\begin{document}
\global\long\def\norm#1{\left\Vert #1\right\Vert }
\global\long\def\bbT{\mathbb{T}}
\global\long\def\bbZ{\mathbb{Z}}
\global\long\def\bbN{\mathbb{N}}
\global\long\def\bbS{\mathcal{S}}
\global\long\def\eps{\varepsilon}
\global\long\def\floor#1{\left\lfloor #1\right\rfloor }
\global\long\def\ceiling#1{\left\lceil #1\right\rceil }
\global\long\def\bbR{\mathbb{R}}

\title{Optimal arithmetic structure in exponential Riesz sequences}

\author{Itay Londner}
\begin{abstract}
We consider exponential systems $E\left(\Lambda\right)=\left\{ e^{i\lambda t}\right\} _{\lambda\in\Lambda}$
for $\Lambda\subset\bbZ$. It has been previously shown by Londner
and Olevskii in \cite{LonOle} that there exists a subset of the circle,
of positive Lebesgue measure, so that every set $\Lambda$ which contains,
for arbitrarily large $N$, an arithmetic progressions of length $N$
and step $\ell=O\left(N^{\alpha}\right)$, $\alpha<1$, cannot be
a Riesz sequence in the $L^{2}$ space over that set. On the other
hand, every set admits a Riesz sequence containing arbitrarily long
arithmetic progressions of length $N$ and step $\ell=O\left(N\right)$.

In this paper we show that every set $\mathcal{S}\subset\mathbb{T}$
of positive measure belongs to a unique class, defined through the
optimal growth rate of the step of arithmetic progressions with respect
to the length that can be found in Riesz sequences in the space $L^{2}\left(\mathcal{S}\right)$.
We also give a partial geometric description of each class.
\end{abstract}

\thanks{The author would like to thank Alexander Olevskii for suggesting
this project, and to Izabella \L aba, Malabika Pramanik and Josh Zahl
for numerous useful conversations on the subject of this paper. The author would also like to express his gratitude to the anonymous referees of this paper for their thorough remarks and suggestions.}

\address{Department of Mathematics, University of British Columbia, Vancouver,
BC, Canada V6T 1Z2}

\email{itayl@math.ubc.ca}


\maketitle

\section{Introduction}

\subsection{Preliminaries}

Let $\bbS\subset\bbR$ be a bounded set of positive Lebesgue measure
($\left|\bbS\right|>0$), and $\Lambda\subset\bbR$ a uniformly discrete
set. The exponential system $E\left(\Lambda\right)=\left\{ e^{i\lambda t}\right\} _{\lambda\in\Lambda}$
is called a \emph{Riesz sequence} in $L^{2}\left(\bbS\right)$ (we
denote $\Lambda\in RS\left(\bbS\right)$) if there exist positive
constants $A,B$ such that 
\begin{equation}
A\sum_{\Lambda}\left|a_{\lambda}\right|^{2}\leq\norm{\sum_{\Lambda}a_{\lambda}e^{i\lambda t}}_{L^{2}\left(\bbS\right)}^{2}\leq B\sum_{\Lambda}\left|a_{\lambda}\right|^{2}\label{eq:RS}
\end{equation}
for every finite sequence of coefficients $\left\{ a_{\lambda}\right\} $. 

Any constant satisfying the left inequality in $\left(\ref{eq:RS}\right)$
will be called a \emph{lower Riesz bound} for $E\left(\Lambda\right)$
in $L^{2}\left(\bbS\right)$.

The problem of determining the exact relationship between sets of
frequencies and their corresponding spectra, in general, is very difficult. Heuristically, it can be stated
that for a fixed spectrum $\mathcal{S}$, the more dense a set of
points is, the harder it becomes to verify that it is a Riesz sequence.
Analogously, for a fixed set of frequencies $\Lambda$, the more complicated
the structure of the spectrum, and the smaller its measure, the less
likely it is for $E\left(\Lambda\right)$ to be a Riesz sequence for
this spectrum. In past years, a large body of work has been dedicated
to the analysis of the interplay between $\bbS$ and $\Lambda$ (see,
for instance, \cite{Chr,OleUla,You} and the references therein).
Below we survey some of the most relevant results.

Throughout this paper, unless stated otherwise, we will always assume
$\bbS$ is a subset of the circle group $\bbT=\bbR/2\pi\bbZ$, which will be identified with the interval $\left[0,2\pi\right)$, and
$\Lambda\subset\bbZ$. Also, in various places, $c$ and $C$ will
denote absolute constants which might be different from one another,
even within the same line.

\subsection{Density}

The case where $\bbS=I$ is an interval is classical. In order to
verify that $E\left(\Lambda\right)$ is indeed a Riesz sequence in
$L^{2}\left(I\right)$, one essentially needs to know the upper Beurling
density of $\Lambda$. 
\begin{thm*}[Kahane, \cite{Kah}]
 Let $\Lambda\subset\bbR.$ If 
\[
D^{+}\left(\Lambda\right):=\lim_{r\rightarrow\infty}\max_{a\in\bbR}\frac{\#\left(\Lambda\cap\left(a,a+r\right)\right)}{r}<\frac{\left|I\right|}{2\pi}
\]
then $\Lambda\in RS\left(I\right)$, while if $D^{+}\left(\Lambda\right)>\frac{\left|I\right|}{2\pi}$
then it is not.
\end{thm*}
For arbitrary sets $\mathcal{S}$ of positive (and finite) measure,
the situation is much more complicated and only a necessary condition
exists. 
\begin{thm*}[Landau, \cite{Lan}]
 Let $\Lambda\subset\bbR.$ If $\Lambda\in RS\left(\bbS\right)$
then $D^{+}\left(\Lambda\right)\leq\frac{\left|\bbS\right|}{2\pi}$.
\end{thm*}

\subsection{Simultaneous Riesz sequences}

With regards to the Riesz sequence property it is only natural to
ask whether there exists a set $\Lambda$ such that $\Lambda\in RS\left(\bbS\right)$
simultaneously for all subsets $\bbS\subset\bbT$ of positive measure.
We call such sequences \emph{simultaneous Riesz sequences.} By Landau's
necessary condition, it is clear that such sets must have zero upper
density. 

Zygmund essentially proved the following
\begin{thm*}[\cite{Zyg}]
 Let $\Lambda=\left\{ \lambda_{j}\right\} \subset\bbZ$ be a Hadamard
lacunary sequence, i.e. for some $q>1$ we have 
\[
\nicefrac{\lambda_{j+1}}{\lambda_{j}}\geq q>1\qquad\forall j,
\]
then $\Lambda$ is a simultaneous Riesz sequence.
\end{thm*}
Later, this result has been further generalized (see \cite{Mih}).

\subsection{Riesz sequences with positive density}

Since a Riesz sequence cannot be too dense, one may ask whether a given
set $\bbS$ admits a Riesz sequence $E\left(\Lambda\right)$ such
that $\Lambda$ has positive density. This question may be considered
under various notions of density. Bourgain and Tzafriri, as a consequence
of their \textquotedblleft restricted invertibility\textquotedblright{}
theorem, answered this question positively proving 
\begin{thm*}[\cite{BouTz}]
 Every set $\bbS\subset\bbT$ of positive measure admits a Riesz
sequence $E\left(\Lambda\right)$, $\Lambda\subset\bbZ$, with positive
asymptotic density

\[
\text{dens}\left(\Lambda\right):=\lim_{r\rightarrow\infty}\frac{\#\left(\Lambda\cap\left(-r,r\right)\right)}{2r}>C\left|\bbS\right|,
\]
where $C$ is an absolute positive constant, independent of $\bbS$.
\end{thm*}

\subsection{Syndetic Riesz sequences}

An even stronger notion is the following; we say that a set $\Lambda$
is syndetic if $\Lambda+\left\{ 0,1,\ldots,d\right\} =\bbZ$, for
some $d\in\bbN$. 

Given a set $\bbS$, Lawton proved (\cite{Law}) that the existence
of a syndetic Riesz sequence is equivalent to the Feichtinger conjecture
for exponentials. The Feichtinger conjecture in its general form has
been proved by Casazza et al (\cite{Cas}) to be equivalent to the
Kadison-Singer problem. The latter has been solved recently by Marcus,
Speilman and Srivastava (\cite{MSS}), and so the existence of syndetic
Riesz sequences holds unconditionally. 

It should be mentioned that the aforementioned solution was used in
\cite{BowLon} to prove the existence of a syndetic Riesz sequence
with a sharp asymptotic bound on the the quantity $d$.

\subsection{Riesz sequences and arithmetic progressions}

In this paper we restrict our attention to the situation in which
the set of frequencies $\Lambda\subset\bbZ$ contains arbitrarily
long arithmetic progressions. More accurately, suppose that for some
increasing sequence $N_{j}$, one can find $\ell_{j}\in\bbN$ and
$M_{j}\in\bbZ$ such that 
\[
M_{j}+\left\{ \ell_{j},2\ell_{j},\ldots,N_{j}\ell_{j}\right\} \subset\Lambda\qquad\forall j
\]
What can be said about the different sets $\bbS$ for which $E\left(\Lambda\right)$
is, or is not a Riesz sequence in $L^{2}\left(\bbS\right)$?
\begin{rem*}
From hereon, given an arithmetic progression, we will denote by $N$
its length, and its step by $\ell$.
\end{rem*}
Since every set $\bbS$ admits a Riesz sequence of positive density,
in particular, it follows from Szemeredi's theorem (\cite{Sze}) that
this set contains arbitrarily long arithmetic progressions. We emphasize
that the existence of Riesz sequences containing arbitrarily long
arithmetic progressions may also be proved independently, we leave
it as an exercise for the curious reader.

On the other hand Miheev proved
\begin{thm*}[\cite{Mih}]
 Given a set $\Lambda$ which contains arbitrarily long arithmetic
progressions, there exists a set $\bbS$ of positive measure such
that $\Lambda\notin RS\left(\bbS\right)$, i.e. $\Lambda$ cannot
be a simultaneous Riesz sequence. 
\end{thm*}
In case $\ell$ grows slowly with respect to $N$, one can choose
the set $\bbS$ independently of $\Lambda$. This question was considered
initially by Bownik and Speegle who gave a quantified version of this
result.
\begin{thm*}[\cite{BowSp}]
There exists a set $\bbS$ such that $E\left(\Lambda\right)$ is
not a Riesz sequence in $L^{2}\left(\bbS\right)$ whenever $\Lambda$
contains arithmetic progressions of length $N_{j}$ and step 
\[
\ell_{j}=o\left(N_{j}^{\nicefrac{1}{2}}\log^{-3}N_{j}\right)\qquad\left(N_{1}<N_{2}<\ldots\right).
\]
\end{thm*}
Later, using a different approach, this was improved by Londner and
Olevskii.
\begin{thm}[\cite{LonOle}]
There exists a set $\bbS\subset\bbT$ such that if a set $\Lambda\subset\bbZ$
contains arithmetic progressions of length $N\left(=N_{1}<N_{2}<\ldots\right)$
and step $\ell=O\left(N^{\alpha}\right)$, $\alpha<1$, then $E\left(\Lambda\right)$
is not a Riesz sequence in $L^{2}\left(\bbS\right)$.
\end{thm}

It should be mentioned that the set $\bbS$ in Theorem 1 was constructed
independently of the choice of $\alpha$, and with arbitrarily small
measure of the complement.

As indicated earlier, every set of positive measure admits a Riesz
sequence containing arbitrarily long arithmetic progressions. It had
been asked in \cite{LonOle} what can be said about the growth rate
of $\ell$ with respect to $N$ in those systems? By Theorem 1 we
know that, in general, $\ell$ cannot grow sublinearly. In \cite{LonOle}
it was proved that Theorem 1 is essentially sharp.
\begin{thm}[\cite{LonOle}]
Given a set $\bbS\subset\bbT$ of positive measure, there exists
a set $\Lambda\subset\bbZ$ such that $E\left(\Lambda\right)$ is
a Riesz sequence in $L^{2}\left(\bbS\right)$, and for infinitely
many $N$'s $\Lambda$ contains an arithmetic progression of length
$N$ and step $\ell=O\left(N\right)$.
\end{thm}

Theorem 2 was proved by explicitly constructing the set $\Lambda$,
which takes the form 
\[
\Lambda=\bigcup_{j\in\bbN}\left(M_{j}+\left\{ N_{j},2N_{j}\ldots,N_{j}^{2}\right\} \right)
\]
for some specially chosen increasing sequence $\left\{ N_{j}\right\} _{j\in\bbN}$
(depending on $\bbS$) and translations $\left\{ M_{j}\right\} _{j\in\bbN}$,
so in reality we have here $\ell_{j}=N_{j}$.

We note that the latter results were considered in the multidimensional
setting as well (see \cite{Lon}).

\section{Results}

\subsection{The classes $\mathcal{A}\left(\alpha\right)$}

Theorems 1 and 2 suggest that one can distinguish between subsets
of the circle through Riesz sequences containing arbitrarily long
arithmetic progressions. The goal of this paper is to show that this
phenomenon occurs at every scale, i.e. every set admits an optimal
asymptotic growth rate which controls the step in any arithmetic progression
of a given length in Riesz sequences over that set. To this end we
define the following classes of sets.
\begin{defn*}
Let $\alpha\in\left[0,1\right]$. We say that a measurable set $\bbS\subset\bbT$
belongs to the class $\mathcal{A}\left(\alpha\right)$, if $\bbS$
admits a Riesz sequence $E\left(\Lambda\right)$ such that $\Lambda$
contains, for infinitely many $N$'s, an arithmetic progression of
length $N$ and step $\ell=O\left(N^{\alpha}\right)$.
\end{defn*}
It follows directly from the definition that $\mathcal{A}\left(\beta\right)\subseteq\mathcal{A}\left(\alpha\right)$
for all $\beta<\alpha$. We observe that by Theorem 2, the class $\mathcal{A}\left(1\right)$
contains all measurable subsets of the circle. When combined with
Theorem 1, it implies that ${\displaystyle \mathcal{A}\left(1\right)\backslash\bigcup_{\beta<1}\mathcal{A}\left(\beta\right)}$
is nonempty. Also, by a trivial argument, the class $\mathcal{A}\left(0\right)$
contains all sets with nonempty interior (but not only those). 

When $\beta<\alpha$, which inclusions are proper? Can we, in fact,
separate these classes from one another?

Our main result answers this question positively.
\begin{thm}
\label{thm:main thm}For any $\alpha\in\left(0,1\right]$ we have
${\displaystyle \mathcal{A}\left(\alpha\right)\backslash\bigcup_{\beta<\alpha}\mathcal{A}\left(\beta\right)}$
is nonempty.
\end{thm}

Such set $\mathcal{S}\in{\displaystyle \mathcal{A}\left(\alpha\right)\backslash\bigcup_{\beta<\alpha}\mathcal{A}\left(\beta\right)}$
would admit a Riesz sequence containing arbitrarily long arithmetic
progressions of length $N$ and step $\ell=O\left(N^{\alpha}\right)$,
but any system $E\left(\Lambda\right)$ for which the step in arithmetic
progressions of length $N$ in $\Lambda$ grow at a rate bounded by
$C\cdot N^{\beta}$ cannot be Riesz sequences in $L^{2}\left(\mathcal{S}\right)$.
We prove Theorem \ref{thm:main thm} by explicitly constructing a set which admits a Riesz sequence having the desired properties. It is important to emphasize that the set constructed in the proof of Theorem \ref{thm:main thm} can be chosen to have arbitrarily small measure of the complement, independently of $\alpha$.

\subsection{On the structure of sets in $\mathcal{A}\left(\alpha\right)$}

Theorem 3 suggests that the arithmetic structure that can be found
inside a Riesz sequence can help to sort subsets of the circle. It
is therefore in our interest to obtain a description of the structure
of sets in a given class. This is done using the following
\begin{defn*}
Given a set $\bbS\subset\bbT$ with $\left|\bbS\right|>0$ and a positive
integer $\ell$, we define the $\ell$-th multiplicity function $\nu_{\ell,\bbS}:\left[0,\frac{2\pi}{\ell}\right)\rightarrow\left\{ 0,1,\ldots,\ell\right\} $
associated with $\bbS$ 
\[
\nu_{\ell,\bbS}\left(t\right):=\left(\mathbf{1}_{\bbS}*\delta_{\frac{2\pi}{\ell}\bbZ}\right)\left(t\right).
\]
\end{defn*}
Our next result conveys the idea that some type of order inside a
set $\bbS$ is reflected in another type of order in Riesz sequences
this set admits.
\begin{thm}
Given a set $\bbS\subset\bbT$ with $\left|\bbS\right|>0$ and $\alpha\in\left[0,1\right]$.
If there exist constants $c,\delta\in\left(0,1\right)$, which might
depend on $\alpha$, such that for infinitely many values of $\ell\in\bbN$,
the $\ell$-th multiplicity function satisfies 
\begin{equation}
\left|\left\{ t\in\bbT\,|\,\nu_{\ell,\bbS}\left(\nicefrac{t}{\ell}\right)<\delta\ell\right\} \right|<\frac{c}{\ell^{\nicefrac{1}{\alpha}}}\label{eq:mult}
\end{equation}
then $\bbS\in\mathcal{A}\left(\alpha\right)$. Moreover, if $\alpha=0$
then this condition is also necessary.
\end{thm}

Here $\alpha=0$ means $\left|\left\{ \nu_{\ell,\bbS}\left(\nicefrac{t}{\ell}\right)<\delta\ell\right\} \right|=0$.

Simply put, a set $\bbS$ satisfies condition $\left(\ref{eq:mult}\right)$
with a fixed $\delta$ if for infinitely many values of $\ell$ and
for every $t\in\bbT$, up to some exceptional set of small measure,
$\bbS$ contains a $\delta$-proportion of the arithmetic progression
of length $\ell$ and step $\frac{2\pi}{\ell}$, passing through $t$.

\section{The classes $\mathcal{A}\left(\alpha\right)$}

The proof of Theorem 3 is composed of 3 parts. In the first part, we
generalize our construction from Section 3 in \cite{LonOle}, constructing
the set $\bbS_{\alpha}$. In the second, we show that this set does
not belong to any of the \textquotedblleft lower\textquotedblright{}
classes, i.e. $\mathcal{\bbS_{\alpha}\notin A}\left(\beta\right)$
when $\beta<\alpha$. In the third part, we prove that it does belong to $\mathcal{A}\left(\alpha\right)$.

Since the case $\alpha=1$ has already been covered in \cite{LonOle},
we restrict to $\alpha\in\left(0,1\right)$.

\subsection{\label{subsec:Construction-of-the}Construction of the set $\protect\bbS_{\alpha}$}

We start by generalizing the construction from \cite{LonOle}. Let
$\alpha\in\left(0,1\right)$, $\eps\in\left(0,\nicefrac{1}{4}\right)$ and define 
\[
\delta\left(\ell\right)=\frac{c_{0}}{\ell^{\nicefrac{1}{\alpha}}}\quad,\quad\ell=1,2,\ldots
\]
where $c_{0}=c_{0}\left(\alpha,\eps\right)\in\left(0,\eps\right)$ will be a fixed constant, chosen so that $\sum_{\ell=1}^{\infty}2\delta\left(\ell\right)<\eps$.

For every $\ell\in\bbN$ set 
\[
I_{\ell}=\left(-\delta\left(\ell\right),\delta\left(\ell\right)\right)\,,\,\mathcal{I}_{\ell}=\left(-\frac{\delta\left(\ell\right)}{\ell},\frac{\delta\left(\ell\right)}{\ell}\right)
\]
and let $\tilde{I}_{\ell}$ $2\pi$-periodic extension of $I_{\ell}$, i.e.
\[
\tilde{I}_{\ell}=\bigcup_{k\in\bbZ}\left(I_{\ell}+2\pi k\right).
\]
We also define 
\begin{equation}
I_{\left[\ell\right]}=\left(\frac{1}{\ell}\cdot\tilde{I}_{\ell}\right)\cap\left[0,2\pi\right)\text{ and }\bbS_{\alpha}=\bbT\backslash\bigcup_{\ell\in\bbN}I_{\left[\ell\right]}\label{eq:set def}
\end{equation}
and see that 
\[
\nicefrac{\left|\bbS_{\alpha}\right|}{2\pi}\geq1-\sum_{\ell=1}^{\infty}\left|I_{\left[\ell\right]}\right|=1-\sum_{\ell=1}^{\infty}2\delta\left(\ell\right)>1-\eps.
\]
Observe that by definition we may identify $I_{\left[\ell\right]}$
with the following union
\begin{equation}
I_{\left[\ell\right]}=\bigcup_{j=0}^{\ell-1}\left(\frac{2\pi j}{\ell}+\mathcal{I}_{\ell}\right)\label{eq:set I_l}
\end{equation}
considered as a subset of $\bbT$. Since the sequence $\left\{ \delta\left(\ell\right)\right\} $
is decreasing, it is easily seen that in the definition of $\bbS_{\alpha}$,
considering the union $\bigcup_{\ell\in\bbN}I_{\left[\ell\right]}$,
any of the sets $I_{\left[\ell\right]}$ may be replaced by a set
$J_{\left[\ell\right]}$, where $I_{\left[1\right]}=J_{\left[1\right]}$
and 
\begin{equation}
J_{\left[\ell\right]}=\bigcup_{\underset{\gcd\left(j,\ell\right)=1}{j=1}}^{\ell-1}\left(\frac{2\pi j}{\ell}+\mathcal{I}_{\ell}\right)\text{ for }\ell>1.\label{eq:set J_l}
\end{equation}
In some cases the sets $J_{\left[\ell\right]}$ would be easier to
handle. We prove
\begin{lem}
$\bbS_{\alpha}\notin \mathcal{A}\left(\beta\right)$ for any $\beta\in\left[0,\alpha\right)$.
\end{lem}

\begin{proof}
Fix $\beta\in\left[0,\alpha\right)$ and let $\Lambda\subset\bbZ$
be such that one can find arbitrarily large $N\in\bbN$ for which
\[
M+\left\{ \ell,\ldots,N\ell\right\} \subset\Lambda
\]
with some $M=M\left(N\right)\in\bbZ$, $\ell=\ell\left(N\right)\in\bbN$
and 
\begin{equation}
\ell<C\left(\beta\right)N^{\beta}.\label{eq:growth}
\end{equation}

Recall that by $\left(\ref{eq:set def}\right)$ we have $t\in I_{\left[\ell\right]}$
if and only if $t\ell\in\tilde{I}_{\ell}\cap\left[0,2\pi\ell\right)$.
Since $\bbS_{\alpha}$ lies inside the complement of $I_{\left[\ell\right]}$,
we get 
\[
\intop_{\bbS_{\alpha}}\left|\sum_{k=1}^{N}a\left(k\right)e^{i\left(M+k\ell\right)t}\right|^{2}dt\leq\intop_{\bbT\backslash I_{\left[\ell\right]}}\left|\sum_{k=1}^{N}a\left(k\right)e^{i\left(M+k\ell\right)t}\right|^{2}dt=
\]
\[
=\intop_{\left[0,2\pi\ell\right)\backslash\tilde{I}_{\ell}}\left|\sum_{k=1}^{N}a\left(k\right)e^{ik\tau}\right|^{2}\frac{d\tau}{\ell}=\intop_{\bbT\backslash I_{\ell}}\left|\sum_{k=1}^{N}a\left(k\right)e^{ik\tau}\right|^{2}d\tau
\]
Next we set $P_{N}\left(t\right)=\frac{1}{\sqrt{N}}\sum_{k=1}^{N}e^{ikt}$,
so $\norm{P_{N}}_{L^{2}\left(\bbT\right)}=1$. Moreover, for every
$t\in \left[0,\pi\right)$ we have $\left|P_{N}\left(t\right)\right|\leq\frac{1}{\sqrt{N}\left|\sin\frac{t}{2}\right|}$,
hence, by $\left(\ref{eq:growth}\right)$ and the fact that for $t\in\bbT$ $P_{N}(t)=P_{N}(2\pi-t)$ we get
\[
\intop_{\bbT\backslash I_{\ell}}\left|P_{N}\left(t\right)\right|^{2}dt\leq\frac{C}{N}\intop_{\delta\left(\ell\right)}^{\pi}\frac{dt}{\sin^{2}\frac{t}{2}}<\frac{C}{N}\intop_{\delta\left(\ell\right)}^{\pi}\frac{dt}{t^{2}}<\frac{C}{\delta\left(\ell\right)N}\leq\frac{C\left(\beta\right)}{\delta\left(\ell\right)\ell^{\nicefrac{1}{\beta}}}
\]
where the last inequality holds for every $N$ for which $\left(\ref{eq:growth}\right)$
holds. It now follows from the choice of $\delta\left(\ell\right)$
that the last term can be made arbitrarily small, and so $E\left(\Lambda\right)$
is not a Riesz sequence in $L^{2}\left(\bbS_{\alpha}\right)$.
\end{proof}

\subsection{Uniting blocks}

In this section we lay out a basic, yet rather useful, principle that
enables the construction of Riesz sequences containing arbitrarily
long arithmetic progressions. Whenever this principle will be applied,
it will always lead to the conclusion that some countable union of
finite arithmetic progressions, which we sometimes refer to as "blocks", form a Riesz sequence over some set.

We require the following lemma.
\begin{lem}[\cite{LonOle}, Lemma 8]
\label{lem:finite}Let $\gamma>0$, $\bbS\subset\bbT$ with $\left|\bbS\right|>0$
and $A_{1},\,A_{2}$ finite sets of integers such that $\gamma$ is
a lower Riesz bound in $L^{2}\left(\bbS\right)$ for $E\left(A_{k}\right)$,
$k=1,2$. Then for any $0<\gamma^{\prime}<\gamma$ there exists $M\in\bbZ\quad$such
that the system $E\left(A_{1}\cup\left(M+A_{2}\right)\right)$ is
a Riesz sequence in $L^{2}\left(\bbS\right)$ with lower Riesz bound
$\gamma^{\prime}$.
\end{lem}

Given a countable collection of finite sets $\left\{ A_{k}\right\} _{k\in\bbN}$
with a common lower Riesz bound, positioning them distant enough from
one another, we can take their union while keeping the Riesz sequence
property.
\begin{cor}
\label{cor:uniting blocks}Let $\bbS\subset\bbT$ with $\left|\bbS\right|>0$
and $\left\{ A_{k}\right\} _{k\in\bbN}$ be a collection of finite
sets of integers. Suppose that $E\left(A_{k}\right)$ is a Riesz sequence
in $L^{2}\left(\bbS\right)$, with lower Riesz bound $\gamma$ for
all $k$. Then there exists a sequence of integers $\left\{ M_{k}\right\} _{k\in\bbN}$
such that the system $E\left(A\right)$, where ${\displaystyle A=\bigcup_{k\in\bbN}\left(M_{k}+A_{k}\right)}$,
is a Riesz sequence in $L^{2}\left(\bbS\right)$ with lower Riesz
bound $\frac{\gamma}{2}$. 
\end{cor}

\begin{proof}
By induction. Take $M_{1}=0$.

Suppose we have chosen $M_{1},\ldots,M_{K}$ such that for 
\[
A\left(K\right)=\bigcup_{k=1}^{K}\left(M_{k}+A_{k}\right)
\]
the corresponding exponential system $E\left(A\left(K\right)\right)$
has lower Riesz bound $\frac{\gamma}{2}\left(1+\frac{1}{K}\right)$. 

Note that $A\left(K\right)$ and $A_{K+1}$ are both finite sets satisfying
the assumptions of Lemma \ref{lem:finite} with $\tilde{\gamma}=\frac{\gamma}{2}\left(1+\frac{1}{K}\right)$.
It follows that we can find $M_{K+1}$ such that $E\left(A\left(K+1\right)\right)$
is a Riesz sequence in $L^{2}\left(\bbS\right)$ with lower Riesz
bound $\frac{\gamma}{2}\left(1+\frac{1}{K+1}\right)<\tilde{\gamma}$.

For the upper Riesz bound, i.e. the right inequality of $\left(\ref{eq:RS}\right)$. Since $A\subset \bbZ$ and $\bbS\subset\bbT$, this inequality holds with $B=1$ and so $E(A)$ is indeed a Riesz sequence in $L^2(\bbS)$ with the desired lower bound.
\end{proof}  

\subsection{\label{subsec:Proving}Proving $\protect\bbS_{\alpha}\in\mathcal{A}\left(\alpha\right)$}

By Corollary \ref{cor:uniting blocks}, in order to prove Theorem
\ref{thm:main thm} it's enough to prove the following 
\begin{lem}
\label{lem:main lem} Let $\alpha\in\left(0,1\right)$ and $\eps\in\left(0,\nicefrac{1}{4}\right)$. For $\mathcal{S}_{\alpha}$, the set constructed in section \ref{subsec:Construction-of-the},
there exist $\mathcal{N}\in\bbN$ and $\gamma>0$, both depending on $\alpha$ and $\varepsilon$, such that for every prime $p>\mathcal{N}$, the exponential system $E\left(B_{p}^{\alpha}\right)$ is a Riesz sequence in $L^{2}\left(\bbS_{\alpha}\right)$ with lower Riesz bound $\gamma$. Here we define 
\[
B_{p}^{\alpha}:=\left\{ p,2p,\ldots,N_{p}p\right\} ,\qquad N_{p}=N\left(\alpha,p\right)=\floor{p^{\nicefrac{1}{\alpha}}}.
\]
\end{lem}

Throughout this section, $\alpha$ and $\eps$ will be fixed. Let $c_{0}$ be the constant chosen in Section \ref{subsec:Construction-of-the}, and let 
\[
Q_{N_{p}}\left(t\right)=\sum_{k=1}^{N_{p}}a\left(k\right)e^{ikt}\quad,\quad\sum_{k=1}^{N_{p}}\left|a\left(k\right)\right|^{2}=1
\]
so that $Q_{N_{p}}\left(pt\right)$ has spectrum in $B_{p}^{\alpha}$.
We prove Lemma \ref{lem:main lem} by bounding from above each of
the integrals 
\[
\intop_{J_{\left[\ell\right]}}\left|Q_{N_{p}}\left(pt\right)\right|^{2}dt,\qquad\ell\in\bbN
\]
separately. Given that our estimates sum up to a quantity smaller than
1 (in fact, it will depend on $\eps$ and $\mathcal{N}$), we will
write
\[
\intop_{\bbS_{\alpha}}\left|Q_{N_{p}}\left(pt\right)\right|^{2}dt\geq1-\sum_{\ell}\intop_{J_{\left[\ell\right]}}\left|Q_{N_{p}}\left(pt\right)\right|^{2}dt.
\]

For convenience we shall denote $f_{\ell}\left(t\right)=\mathbf{\mathbf{\mathds{1}}}_{I_{\left[\ell\right]}}\left(t\right)$
the characteristic function of the set $I_{\left[\ell\right]}$. Notice
that since $f_{\ell}$ is a $\frac{2\pi}{\ell}$-periodic function,
we have 
\begin{equation}\label{eq:nonzero coeff}
\widehat{f_{\ell}}\left(n\right)\neq0\Rightarrow n\in\ell\bbZ.
\end{equation}
We will require the following lemmas.

\begin{lem}
	Given $\alpha\in\left(0,1\right)$ and a prime number $p$, we have
	\[
	\sum_{\ell=1}^{\floor{c_{0}p}}\intop_{J_{\left[\ell\right]}}\left|Q_{N_{p}}\left(pt\right)\right|^{2}dt\leq \sum_{\ell=1}^{\floor{c_{0}p}}\left|I_{\left[\ell\right]}\right|+2c_{0}.	
	\]
\end{lem}
\begin{proof}
For $1\leq\ell\leq\floor{c_{0}p}$ we write 
\[
\intop_{J_{\left[\ell\right]}}\left|Q_{N_{p}}\left(pt\right)\right|^{2}dt\leq\intop_{I_{\left[\ell\right]}}\left|Q_{N_{p}}\left(pt\right)\right|^{2}dt=\intop_{pI_{\left[\ell\right]}}\left|Q_{N_{p}}\left(\tau\right)\right|^{2}\frac{d\tau}{p}
\]
We use the following key observation: considering the change of variables
$pt=\tau$ as above, we notice that if, as in (\ref{eq:set I_l}),
we identify $I_{\left[\ell\right]}$ as a subset of $\left[0,2\pi\right)$
which is a union of $\ell$ disjoint intervals (counting $\left[0,\nicefrac{\delta\left(\ell\right)}{\ell}\right)\cup\left(2\pi-\nicefrac{\delta\left(\ell\right)}{\ell},2\pi\right)$
as one interval), each of length $\frac{2\delta\left(\ell\right)}{\ell}$,
and centered at points $0,\frac{2\pi}{\ell},\ldots\frac{2\pi\left(\ell-1\right)}{\ell}$.
We get that $pI_{\left[\ell\right]}$ is a subset of $\left[0,2\pi p\right)$
composed of $\ell$ disjoint intervals, each of length $p\frac{2\delta\left(\ell\right)}{\ell}$,
and centered at points $0,\frac{2\pi p}{\ell},\ldots\frac{2\pi p\left(\ell-1\right)}{\ell}$. 

Since the mapping from $\bbZ/\ell\bbZ$ to itself defined by $j\mapsto pj$
is an isomorphism (this is so since $p$ is invertible $\mod\ell$),
residues modulo $\ell$ are mapped to themselves under this mapping
and so we conclude that there exists a permutation $\sigma$ of $\left\{ 0,1,\ldots,\ell-1\right\} $
such that $\sigma\left(0\right)=0$ and 
\[
 \frac{pj}{\ell} \equiv \frac{\sigma\left(j\right)}{\ell} \text{ mod 1},\quad j\in\left\{ 1,\ldots,\ell-1\right\} .
\]
Since $Q_{N_{p}}$ is a trigonometric polynomial of period $2\pi$,
we get 
\[
\intop_{\frac{2\pi jp}{\ell}+p\mathcal{I}_{\ell}}\left|Q_{N_{p}}\left(t\right)\right|^{2}dt=\intop_{\frac{2\pi\sigma\left(j\right)}{\ell}+p\mathcal{I}_{\ell}}\left|Q_{N_{p}}\left(t\right)\right|^{2}dt,\;j\in\left\{ 0,1,\ldots,\ell-1\right\} .
\]
By the definition of $\sigma$ as a permutation we have 
\[
\mathcal{U_\ell}=\left\{\frac{2\pi\sigma\left(j\right)}{\ell}+p\mathcal{I}_{\ell}\,\bigg|\,j=0,\ldots,\ell-1\right\}=\left\{\frac{2\pi j}{\ell}+p\mathcal{I}_{\ell}\,\bigg|\,j=0,\ldots,\ell-1\right\}.
\]
We notice that the collection of intervals in $\mathcal{U}_\ell$ covers every point of the circle at most $\floor{2p\delta\left(\ell\right)}+2$ times. Indeed let $t\in\bbT$, since all the intervals are of the same length and their centers are evenly spaced on the circle, we have
\[
\#\left\{j\in\{0,\ldots,\ell\}\,\bigg|\,t\in\left(\frac{2\pi j}{\ell}+p\mathcal{I}_{\ell}\right)\right\}=
\]
\[
=\#\left\{j\in\{0,\ldots,\ell\}\,\bigg|\,\frac{2\pi j}{\ell}\in\left( t+p\mathcal{I}_{\ell}\right)\right\}\leq \floor{\ell\cdot\frac{\left|p\mathcal{I}_\ell\right|}{2\pi}}+2=\floor{2p\delta\left(\ell\right)}+2.
\]
From this we get 
\[
\intop_{pI_{\left[\ell\right]}}\left|Q_{N_{p}}\left(\tau\right)\right|^{2}\frac{d\tau}{p}\leq\frac{1}{p}\left(\floor{2p\delta\left(\ell\right)}+2\right)\leq2\delta\left(\ell\right)+\frac{2}{p}.
\]
summing over all $\ell$ in the current range, we get 
\[
\sum_{\ell=1}^{\floor{c_{0}p}}\intop_{J_{\left[\ell\right]}}\left|Q_{N_{p}}\left(pt\right)\right|^{2}dt\leq\sum_{\ell=1}^{\floor{c_{0}p}}\intop_{pI_{\left[\ell\right]}}\left|Q_{N_{p}}\left(\tau\right)\right|^{2}\frac{d\tau}{p}\leq
\]

\begin{equation}
\leq\sum_{\ell=1}^{\floor{c_{0}p}}\left(2\delta\left(\ell\right)+\frac{2}{p}\right)<\sum_{\ell=1}^{\floor{c_{0}p}}\left|I_{\left[\ell\right]}\right|+2c_{0}.\label{eq:estimate 1}
\end{equation}
\end{proof}
Denote $A=\left\{ \ceiling{c_{0}p}\leq\ell\leq N_{p}\,|\,p\nmid\ell\right\} $.
\begin{lem}\label{lem:main}
Given $\alpha$ and $\eps$, there exists $\mathcal{N}\in\bbN$ such that for any prime number $p>\mathcal{N}$ we have 
\[
	\sum_{\ell\in A}\intop_{J_{\left[\ell\right]}}\left|Q_{N_{p}}\left(pt\right)\right|^{2}dt\leq \eps
\]
\end{lem}

The proof of the Lemma \ref{lem:main} will be divided to two cases: $\alpha\in\left(0,\nicefrac{1}{2}\right)$ and $\alpha\in\left[\nicefrac{1}{2},1\right)$.
We begin with some intuition. Write 
\[
\sum_{\ell\in A}\intop_{J_{\left[\ell\right]}}\left|Q_{N_{p}}\left(pt\right)\right|^{2}dt=\sum_{\ell\in A}\intop_{pJ_{\left[\ell\right]}}\left|Q_{N_{p}}\left(\tau\right)\right|^{2}\frac{d\tau}{p}
\]
by the same reasoning as we had for the previous range of $\ell$'s, since for all $\ell\in A$ we have $\gcd\left(\ell,p\right)=1$, we can conclude
that 
\begin{equation}
\sum_{\ell\in A}\intop_{pJ_{\left[\ell\right]}}\left|Q_{N_{p}}\left(\tau\right)\right|^{2}\frac{d\tau}{p}=\frac{1}{p}\sum_{\ell\in A}\intop_{\mathcal{J}_{p,\left[\ell\right]}}\left|Q_{N_{p}}\left(\tau\right)\right|^{2}d\tau\label{eq:Var change}
\end{equation}
where
\[
\mathcal{J}_{p,\left[\ell\right]}=\bigcup_{\underset{\gcd\left(j,\ell\right)=1}{j=1}}^{\ell-1}\left(\frac{2\pi j}{\ell}+p\mathcal{I}_{\ell}\right).
\]
Notice that the sets $\mathcal{J}_{p,\left[\ell\right]}$ may overlap,
and so one way of estimating the integrals in the above sum is by
bounding, for every $\tau\in\bbT$, the number of sets $\left\{ \mathcal{J}_{p,\left[\ell\right]}\right\} _{\ell\in A}$
to which $\tau$ belongs. We will show that when $0<\alpha<\nicefrac{1}{2}$
this number is bounded independently of $p$, i.e. every $\tau$ belongs
to at most $d$ of the sets $\left\{ \mathcal{J}_{p,\left[\ell\right]}\right\} _{\ell\in A}$,
and $d$ depends only on $\alpha$. While for $\nicefrac{1}{2}\leq\alpha<1$
we shall prove that even though this number depends on $p$, it grows
slow enough to allow control over the latter sum. Moreover, when $\alpha\in\left(0,\nicefrac{1}{2}\right)$
note that 
\[
\sum_{\ell\in A}\left|\mathcal{J}_{p,\left[\ell\right]}\right|\leq p\sum_{\ell=\ceiling{c_{0}p}}^{N_{p}}\left|I_{\left[\ell\right]}\right|=c_{0}p\sum_{\ell=\ceiling{c_{0}p}}^{N_{p}}\frac{1}{\ell^{\nicefrac{1}{\alpha}}}\leq c_{0}p\frac{1}{\left(c_{0}p\right)^{\left(\nicefrac{1}{\alpha}-1\right)}}=o\left(1\right)\text{ as } p\rightarrow\infty
\]
This observation should serve as an intuitive justification for
Lemma \ref{lem:disjoint} below.

In addition, for every $\ell\in A$ the intervals forming the set
$\mathcal{J}_{p,\left[\ell\right]}$ are pairwise disjoint, that is
\begin{equation}
\left(\frac{2\pi j_{1}}{\ell}+p\mathcal{I}_{\ell}\right)\cap\left(\frac{2\pi j_{2}}{\ell}+p\mathcal{I}_{\ell}\right)=\emptyset\text{ for all }1\leq j_{1}<j_{2}\leq\ell-1.\label{eq:3.6}
\end{equation}
Indeed $\left(\ref{eq:3.6}\right)$ holds if and only if
\[
\frac{p\delta\left(\ell\right)}{\ell}\leq\frac{\pi}{\ell}\iff\left(cp\right)^{\alpha}\leq\ell,
\]
which holds for all $\ell\in A$. Moreover, the following holds true.
\begin{lem}
\label{lem:disjoint}Let $\alpha\in\left(0,\nicefrac{1}{2}\right),\eta>\frac{\alpha}{1-\alpha}$
and set $\eta^{\prime}=\frac{2}{1+\alpha}\left(\eta-\nicefrac{\alpha}{2}\right)$.
Then for all 
\begin{equation}
p^{\eta}\leq\ell_{1}<\ell_{2}\leq p^{\eta^{\prime}}\,,\,p\nmid\ell_{1},\ell_{2}\label{eq:cond}
\end{equation}
we have $\mathcal{J}_{p,\left[\ell_{1}\right]}\cap\mathcal{J}_{p,\left[\ell_{2}\right]}=\emptyset$.
\end{lem}

\begin{proof}
Fix $\alpha\in\left(0,\nicefrac{1}{2}\right),\eta>\frac{\alpha}{1-\alpha}$
and let $\ell_{1},\ell_{2}$ satisfy $\left(\ref{eq:cond}\right)$.
First notice that 
\[
\eta^{\prime}=\frac{2}{1+\alpha}\left(\eta-\nicefrac{\alpha}{2}\right)>\eta\iff\eta>\frac{\alpha}{1-\alpha}.
\]
Next, the sets $\mathcal{J}_{p,\left[\ell_{1}\right]}$ and $\mathcal{J}_{p,\left[\ell_{2}\right]}$
are disjoint if and only if for all $j_{k}\in\left\{ 1,\ldots\ell_{k}-1\right\} $
with $\gcd\left(j_{k},\ell_{k}\right)=1$, for $k=1,2$, we have 
\[
\frac{p\delta\left(\ell_{1}\right)}{\ell_{1}}+\frac{p\delta\left(\ell_{2}\right)}{\ell_{2}}\leq2\pi\left|\frac{j_{1}}{\ell_{1}}-\frac{j_{2}}{\ell_{2}}\right|=2\pi\frac{\left|j_{1}\ell_{2}-j_{2}\ell_{1}\right|}{\ell_{1}\ell_{2}},
\]
since $\left|j_{1}\ell_{2}-j_{2}\ell_{1}\right|$ is a positive integer,
it's enough to verify 
\[
\frac{p\delta\left(\ell_{1}\right)}{\ell_{1}}+\frac{p\delta\left(\ell_{2}\right)}{\ell_{2}}\leq\frac{2\pi}{\ell_{1}\ell_{2}}\iff\frac{c_{0}}{2\pi}p\left(\frac{\ell_{2}}{\ell_{1}^{\nicefrac{1}{\alpha}}}+\frac{\ell_{1}}{\ell_{2}^{\nicefrac{1}{\alpha}}}\right)\leq1.
\]
By plugging the assumption $\left(\ref{eq:cond}\right)$, we get that
\[
c_{0}p\frac{\ell_{2}^{1+\nicefrac{1}{\alpha}}+\ell_{1}^{1+\nicefrac{1}{\alpha}}}{\left(\ell_{1}\ell_{2}\right)^{\nicefrac{1}{\alpha}}}\leq c_{0}p^{\eta^{\prime}\left(1+\nicefrac{1}{\alpha}\right)+1-\nicefrac{2\eta}{\alpha}}\leq c_{0},
\]
and the last inequality holds true whenever
\[
\eta^{\prime}\left(1+\nicefrac{1}{\alpha}\right)+1\leq\nicefrac{2\eta}{\alpha}\iff\eta^{\prime}\leq\frac{2}{1+\alpha}\left(\eta-\nicefrac{\alpha}{2}\right).
\]
\end{proof}
As a corollary, we can break the range $A$ into finitely many sub-ranges,
so that within every sub-range the sets $\mathcal{J}_{p,\left[\ell\right]}$
are pairwise disjoint.
\begin{cor}
\label{cor: disjoint}Let $\alpha\in\left(0,\nicefrac{1}{2}\right)$.
Then there exist  positive integers $\mathcal{N}$, $d=d\left(\alpha\right)$ and
a sequence $\eta_{1}<\eta_{2}<\ldots<\eta_{d}$ satisfying 
\end{cor}

\begin{enumerate}
\item $p^{\eta_{1}}<c_{0}p$ for all primes $p>\mathcal{N}$
\item $\eta_{d}>\nicefrac{1}{\alpha}$
\end{enumerate}
so that for every $i\in\left\{ 1,\ldots,d-1\right\} $ and every $p^{\eta_{i}}\leq\ell_{1}<\ell_{2}\leq p^{\eta_{i+1}}$
satisfying $p\nmid\ell_{1},\ell_{2}$, the sets $\mathcal{J}_{p,\left[\ell_{1}\right]}$
and $\mathcal{J}_{p,\left[\ell_{2}\right]}$ are disjoint.
\begin{proof}
Fix $\alpha\in\left(0,\nicefrac{1}{2}\right)$. We prove Corollary
\ref{cor: disjoint} by iterating Lemma \ref{lem:disjoint}, setting
\[
\eta_{i+1}=\frac{2}{1+\alpha}\left(\eta_{i}-\nicefrac{\alpha}{2}\right)\text{ for }i\geq1
\]
one can explicitly compute 
\[
\eta_{i+1}=\left(\frac{2}{1+\alpha}\right)^{i}\left(\eta_{1}-\frac{\alpha}{1-\alpha}\right)+\frac{\alpha}{1-\alpha},\quad i\geq1.
\]
Since $\alpha<\nicefrac{1}{2}$ we may take $\eta_{1}\in\left(\frac{\alpha}{1-\alpha},1\right)$. We get
as a consequence that the sequence $\left\{ \eta_{i}\right\} $ is
monotonically increasing and clearly property $\left(1\right)$ is
satisfied for all primes $p>\mathcal{N}$, given that $\mathcal{N}$ is sufficiently large. Since $\left\{ \eta_{i}\right\} $
grows, roughly, as a geometric sequence we can deduce there exists $d$
such that $\eta_{d}>\nicefrac{1}{\alpha}$. 
\end{proof}

\begin{proof}[Proof of Lemma \ref{lem:main}, $\alpha\in\left(0,\nicefrac{1}{2}\right)$]
In order to get an estimate on the contribution of the sets $\mathcal{J}_{p,\left[\ell\right]}$
when $\ell$ varies through $A$, we write 
\[
A=\bigcup_{i=1}^{d-1}A_{p,i},
\]
where 
\[
A_{p,i}=\left\{ \ceiling{p^{\eta_{i}}}\leq\ell<\floor{p^{\eta_{i+1}}}\text{\,|\,}p\nmid\ell\right\} .
\]
By Corollary \ref{cor: disjoint}
\[
\intop_{\underset{\ell\in A_{p,i}}{\bigcup}\mathcal{J}_{p,\left[\ell\right]}}\left|Q_{N_{p}}\left(\tau\right)\right|^{2}d\tau\leq\intop_{\bbT}\left|Q_{N_{p}}\left(\tau\right)\right|^{2}d\tau=1
\]
Taking into consideration the change of variables $\left(\ref{eq:Var change}\right)$
we deduce 
\begin{equation}
\sum_{\ell\in A}\intop_{J_{\left[\ell\right]}}\left|Q_{N_{p}}\left(pt\right)\right|^{2}dt\leq\frac{1}{p}\sum_{i=1}^{d-1}\intop_{\underset{\ell\in A_{p,i}}{\bigcup}\mathcal{J}_{p,\left[\ell\right]}}\left|Q_{N_{p}}\left(\tau\right)\right|^{2}d\tau\leq\frac{d-1}{p}=o\left(1\right)\text{ as } p\rightarrow\infty\label{eq:estimate 2}
\end{equation}
\end{proof}

When $\alpha\in\left[\nicefrac{1}{2},1\right)$, we need the following.
\begin{defn*}
	Given $\rho\in\left(0,1\right)$, for every $x\in\left[0,1\right]$
	and $N\in\bbN$ we denote by $M_{\rho}\left(x,N\right)$ the function
	that counts the number of pairs of integers $\left(m,n\right)$ satisfying
	\begin{equation}
	\left|x-\frac{m}{n}\right|<\frac{1}{n^{1+\rho}}\label{eq:dioph}
	\end{equation}
	where $1\leq n\leq N$, $1\leq m<n$ and $\gcd\left(m,n\right)=1$. 
\end{defn*}

The counting function $M_{\rho}$ will help us estimate the overlap of the sets $\left\{ \mathcal{J}_{p,\left[\ell\right]}\right\} _{\ell\in A}$ as explained above.

\begin{lem}
\label{lem:counting fn}Let $\rho\in\left(0,1\right)$. Then there
exists $C>0$ such that 
\[
M_{\rho}\left(x,N\right)\leq CN^{1-\rho},\quad x\in\left[0,1\right].
\]
\end{lem}

\begin{proof}
Fix $x\in\left[0,1\right]$. Note that for any two fractions $\frac{m_{1}}{n_{1}},\frac{m_{2}}{n_{2}}$
where $2^{-k}N\leq n_{j}\leq2^{-k+1}N$, $1\leq m_{j}<n_{j}$, $\gcd\left(m_{j},n_{j}\right)=1$
for $j=1,2$ and $k\geq1$ we have 
\[
\left|\frac{m_{1}}{n_{1}}-\frac{m_{2}}{n_{2}}\right|=\frac{\left|m_{1}n_{2}-n_{1}m_{2}\right|}{n_{1}n_{2}}\geq\frac{2^{2k}}{N^{2}}.
\]
Hence the maximal number of pairs $\left(m,n\right)$ with $2^{-k}N\leq n\leq2^{-k+1}N$
and which satisfy (\ref{eq:dioph}) is at most 
\[
\frac{\frac{2\cdot2^{k\left(1+\rho\right)}}{N^{1+\rho}}}{\frac{2^{2k}}{N^{2}}}=2\cdot2^{k\left(1+\rho-2\right)}N^{1-\rho}.
\]
When summing over all dyadic intervals the lemma follows.
\end{proof}
We are now ready to finish the proof of Lemma \ref{lem:main}.

\begin{proof}[Proof of Lemma \ref{lem:main}, $\alpha\in\left[\nicefrac{1}{2},1\right)$]
Going back to the sum from (\ref{eq:Var change})
\[
\frac{1}{p}\sum_{\ell\in A}\intop_{\mathcal{J}_{p,\left[\ell\right]}}\left|Q_{N_{p}}\left(\tau\right)\right|^{2}d\tau,
\]
we wish to apply the estimate from Lemma \ref{lem:counting fn}. Since
$\tau\in\mathcal{J}_{p,\left[\ell\right]}$ if and only if there exists
$1\leq j<\ell$ with $\gcd\left(j,\ell\right)=1$ such that 
\[
\left|\frac{\tau}{2\pi}-\frac{j}{\ell}\right|<\frac{p\delta\left(\ell\right)}{\ell}=\frac{c_{0}p}{\ell^{1+\nicefrac{1}{\alpha}}}\leq\frac{\ell}{\ell^{1+\nicefrac{1}{\alpha}}}=\frac{1}{\ell^{1+\left(\nicefrac{1}{\alpha}-1\right)}}
\]
and that last inequality holds for all $\ell\in A$. Lemma \ref{lem:counting fn} now implies
\[
\frac{1}{p}\sum_{\ell\in A}\intop_{\mathcal{J}_{p,\left[\ell\right]}}\left|Q_{N_{p}}\left(\tau\right)\right|^{2}d\tau\leq\frac{1}{p}\intop_{\bbT}\left|Q_{N_{p}}\left(\tau\right)\right|^{2}M_{\nicefrac{1}{\alpha}-1}\left(\frac{\tau}{2\pi},N_{p}\right)d\tau\leq
\]
\begin{equation}
\leq\frac{1}{p}CN_{p}^{1-\left(\nicefrac{1}{\alpha}-1\right)}=\frac{1}{p}CN_{p}^{2-\nicefrac{1}{\alpha}}\label{eq: estimate 3}
\end{equation}
Observe that 
\[
\frac{1}{p}CN_{p}^{2-\nicefrac{1}{\alpha}}=Cp^{\nicefrac{1}{\alpha}\left(2-\nicefrac{1}{\alpha}\right)-1}=o\left(1\right)\iff0<\alpha^{2}-2\alpha+1=\left(\alpha-1\right)^{2}
\]
which hold for $\alpha\in\left[\nicefrac{1}{2},1\right)$.
\end{proof}

\begin{rem*}
It is evident that the bound attained from Lemma \ref{lem:counting fn}
can be used when $\alpha\in\left(0,\nicefrac{1}{2}\right)$, but what
we have actually proved for this case is that the counting function
is uniformly bounded in both $N$ and $x$.
\end{rem*}

The remaining range, $\ell\geq N_{p}$, will be broken down to two parts. We write 
\[
\intop_{J_{\left[\ell\right]}}\left|Q_{N_{p}}\left(pt\right)\right|^{2}dt\leq\intop_{I_{\left[\ell\right]}}\left|Q_{N_{p}}\left(pt\right)\right|^{2}dt=\sum_{k,k^{\prime}=1}^{N_{p}}a\left(k\right)\overline{a\left(k^{\prime}\right)}\intop_{\bbT}f_{\ell}\left(t\right)e^{-ip\left(k^{\prime}-k\right)t}dt=
\]
\[
=\left|I_{\left[\ell\right]}\right|\sum_{k=1}^{N_{p}}\left|a\left(k\right)\right|^{2}+\sum_{\underset{k\neq k^{\prime}}{k,k^{\prime}=1}}^{N_{p}}a\left(k\right)\overline{a\left(k^{\prime}\right)}\widehat{f_{\ell}}\left(p\left(k^{\prime}-k\right)\right)=
\]
\begin{equation}
=\left|I_{\left[\ell\right]}\right|+\sum_{\underset{k\neq k^{\prime}}{k,k^{\prime}=1}}^{N_{p}}a\left(k\right)\overline{a\left(k^{\prime}\right)}\widehat{f_{\ell}}\left(p\left(k^{\prime}-k\right)\right).\label{eq:open brackets}
\end{equation}

\begin{lem}\label{lem:eqality}
	Given $\alpha$, for any prime $p$ we have 
	\begin{itemize}
		\item
		$\displaystyle \sum_{\underset{p\nmid\ell}{\ell=N_{p}+1}}^{pN_{p}-1}\intop_{I_{\left[\ell\right]}}\left|Q_{N_{p}}\left(pt\right)\right|^{2} dt=\sum_{\underset{p\nmid\ell}{\ell=N_{p}+1}}^{pN_{p}-1}\left|I_{\left[\ell\right]}\right|$
		\item  $\displaystyle\sum_{\ell=pN_{p}}^{\infty}\intop_{I_{\left[\ell\right]}}\left|Q_{N_{p}}\left(pt\right)\right|^{2}dt=\sum_{\ell=pN_{p}}^{\infty}\left|I_{\left[\ell\right]}\right|$
	\end{itemize}
\end{lem}
\begin{proof}
Starting with the second sum, we notice that if $\ell\geq pN_{p}$, in the second term of $\left(\ref{eq:open brackets}\right)$ we have $k$ and $k^{\prime}$ in
$\left\{ 1,\ldots,N_{p}\right\}$ and $k\neq k^{\prime}$,
hence 
\[
p\left(k^{\prime}-k\right)\in\left\{ p,2p,\ldots,\left(N_{p}-1\right)p\right\} ,
\]
which, by $\left(\ref{eq:nonzero coeff}\right)$, implies that $\widehat{f_{\ell}}\left(p\left(k^{\prime}-k\right)\right)=0$
and so 
\begin{equation}
\sum_{\ell=pN_{p}}^{\infty}\intop_{I_{\left[\ell\right]}}\left|Q_{N_{p}}\left(pt\right)\right|^{2}dt=\sum_{\ell=pN_{p}}^{\infty}\left|I_{\left[\ell\right]}\right|.\label{eq:estimate 4}
\end{equation}
This also applies when $N_{p}<\ell<pN_{p}$ and $p\nmid\ell$. Indeed,
assuming without loss of generality that $k^{\prime}>k$, we have
$\widehat{f_{\ell}}\left(p\left(k^{\prime}-k\right)\right)\neq0$
only if 
\[
m\ell=p\left(k^{\prime}-k\right)
\]
with some $m\in\bbN$. Since $p$ does not divide $\ell$ it must
divide $m$, in which case $m=pm^{\prime}$ with some $m^{\prime}\in\bbN$,
and 
\[
m^{\prime}\ell=k^{\prime}-k.
\]
But this is not possible since $m^{\prime}\ell\geq\ell>N_{p}$ while
$k^{\prime}-k<N_{p}$. It follows that 
\begin{equation}
\sum_{\underset{p\nmid\ell}{\ell=N_{p}+1}}^{pN_{p}-1}\intop_{I_{\left[\ell\right]}}\left|Q_{N_{p}}\left(pt\right)\right|^{2}dt=\sum_{\underset{p\nmid\ell}{\ell=N_{p}+1}}^{pN_{p}-1}\left|I_{\left[\ell\right]}\right|\label{eq:estimes 5}
\end{equation}
as required.
\end{proof}
Last, the triangle inequality and Cauchy-Schwarz inequality give the  estimate 
\begin{equation}
\intop_{E}\left|\sum_{k=1}^{N}a\left(k\right)e^{i\lambda_{k}t}\right|^{2}dt\leq\left|E\right|\left(\sum_{k=1}^{N}\left|a\left(k\right)\right|\right)^{2}\leq\left|E\right|N\sum_{k=1}^{N}\left|a\left(k\right)\right|^{2},\label{eq:CS}
\end{equation}
which holds for any measurable set $E$, any $N\in\bbN$, sequence
of scalars $\left\{ a\left(k\right)\right\} _{k=1}^{N}$ and sequence
of real numbers $\left\{ \lambda_{k}\right\} _{k=1}^{N}$.
\begin{cor}
	Given $\alpha$, for any prime number $p$ we have 
	\[
	\sum_{\underset{p|\ell}{\ell=p}}^{pN_{p}}\intop_{J_{\left[\ell\right]}}\left|Q_{N_{p}}\left(pt\right)\right|^{2}dt< c_{0}.
	\]
\end{cor}
\begin{proof}
For $p\leq\ell\le pN_{p}$ and $p|\ell$, applying (\ref{eq:CS}) we get
\[
	\sum_{\underset{p|\ell}{\ell=p}}^{pN_{p}}\intop_{J_{\left[\ell\right]}}\left|Q_{N_{p}}\left(pt\right)\right|^{2}dt\leq
\sum_{\underset{p|\ell}{\ell=p}}^{pN_{p}}\intop_{I_{\left[\ell\right]}}\left|Q_{N_{p}}\left(pt\right)\right|^{2}dt=
\]
\[
=\sum_{j=1}^{N_{p}}\intop_{I_{\left[jp\right]}}\left|Q_{N_{p}}\left(pt\right)\right|^{2}dt\leq N_{p}\sum_{j=1}^{N_{p}}\left|I_{\left[jp\right]}\right|=
\]
\begin{equation}
=N_{p}\sum_{j=1}^{N_{p}}\delta\left(jp\right)=c_{0}N_{p}\sum_{j=1}^{N_{p}}\frac{1}{\left(jp\right)^{\nicefrac{1}{\alpha}}}<c_{0}\frac{N_{p}}{p^{\nicefrac{1}{\alpha}}}\leq c_{0}.\label{eq:estimate 6}
\end{equation}
\end{proof}
We are now in a position to finish the proof of Lemma \ref{lem:main lem}.
\begin{proof}[Proof of Lemma \ref{lem:main lem}]
We gather all our estimates (\ref{eq:estimate 1}), (\ref{eq:estimate 2}),
(\ref{eq: estimate 3}), (\ref{eq:estimate 4}), (\ref{eq:estimes 5})
and (\ref{eq:estimate 6}), which hold for all primes $p>\mathcal{N}$,
we finally have
\[
\intop_{\bbS_{\alpha}}\left|Q_{N_{p}}\left(pt\right)\right|^{2}dt\geq1-\sum_{\ell}\intop_{J_{\left[\ell\right]}}\left|Q_{N_{p}}\left(pt\right)\right|^{2}dt\geq
\]
\[
\geq1-\left(\sum_{\ell}\left|I_{\left[\ell\right]}\right|+3c_{0}+o\left(1\right)\right)\geq
\]
\[
\geq1-\left(\eps+3c_{0}+o\left(1\right)\right)>1-4\eps+o\left(1\right)\text{ as } p\rightarrow \infty
\]
Therefore given $\alpha$ and $\eps$, if we choose $\mathcal{N}$ sufficiently large, we can find some positive constant
$\gamma$ which is a lower Riesz bound for $E\left(B_{p}^{\alpha}\right)$ in $L^{2}\left(\bbS_{\alpha}\right)$, for all primes $p>\mathcal{N}$. By Corollary \ref{cor:uniting blocks}, we indeed get that $\bbS_{\alpha}\in\mathcal{A}(\alpha)$.
\end{proof}

\section{The multiplicity function}

In this section we prove a structural sufficient condition to be a
member of the class $\mathcal{A\left(\alpha\right)}$. Given $\alpha\in(0,1)$ and $\ell,N\in\bbN$, we again use an $L^2$-normalized polynomial
\[
Q_{N}\left(t\right)=\sum_{k=1}^{N}a\left(k\right)e^{ikt}\quad,\quad\sum_{k=1}^{N}\left|a\left(k\right)\right|^{2}=1
\]
so that $Q_{N}\left(\ell t\right)$ has spectrum in $\{\ell,2\ell,\ldots,N\ell\}$.

\begin{proof}[Proof of Theorem 4, the case $\alpha>0$]
Fix $\alpha\in\left(0,1\right]$, $c,\delta\in\left(0,1\right)$. Let  $\mathcal{L}$ be a set of all positive integers for which (\ref{eq:mult}) holds, and let $\ell\in\mathcal{L}$. We set $N=\floor{\ell^{\nicefrac{1}{\alpha}}}$ and notice that 
\[
\intop_{\bbS}\left|Q_{N}\left(\ell t\right)\right|^{2}dt=\intop_{\left[0,\frac{2\pi}{\ell}\right)}\left|Q_{N}\left(\ell t\right)\right|^{2}\nu_{\ell,\bbS}\left(t\right)dt=
\]
making a change of variables we get	
\[
=\intop_{\bbT}\left|Q_{N}\left(\tau\right)\right|^{2}\nu_{\ell,\bbS}\left(\nicefrac{\tau}{\ell}\right)\frac{d\tau}{\ell}\geq\delta\intop_{\left\{ \nu_{\ell,\bbS}\left(\nicefrac{\tau}{\ell}\right)\geq\delta\ell\right\} }\left|Q_{N}\left(\tau\right)\right|^{2}d\tau,
\]
clearly we can write 
\[
\intop_{\left\{ \nu_{\ell,\bbS}\left(\nicefrac{\tau}{\ell}\right)\geq\delta\ell\right\} }\left|Q_{N}\left(\tau\right)\right|^{2}d\tau=1-\intop_{\left\{ \nu_{\ell,\bbS}\left(\nicefrac{\tau}{\ell}\right)<\delta\ell\right\} }\left|Q_{N}\left(\tau\right)\right|^{2}d\tau,
\]
applying the estimate (\ref{eq:CS}) we can bound the last term 
\[
\intop_{\left\{ \nu_{\ell,\bbS}\left(\nicefrac{\tau}{\ell}\right)<\delta\ell\right\} }\left|Q_{N}\left(\tau\right)\right|^{2}d\tau\leq N\left|\left\{ \nu_{\ell,\bbS}\left(\nicefrac{\tau}{\ell}\right)<\delta\ell\right\} \right|
\]
Plugging in the assumption (\ref{eq:mult}), we get
\[
\intop_{\bbS}\left|Q_{N}\left(\ell t\right)\right|^{2}dt\geq\delta\left(1-N\left|\left\{ \nu_{\ell,\bbS}\left(\nicefrac{\tau}{\ell}\right)<\delta\ell\right\} \right|\right)\geq\delta\left(1-c\right)
\]
This implies that for all $\ell\in\mathcal{L}$, the exponential system $E\left(\{\ell,2\ell,\ldots,N\ell\}\right)$ has lower Riesz bound $\delta\left(1-c\right)$. Corollary \ref{cor:uniting blocks} now gives $\bbS\in\mathcal{A}\left(\alpha\right)$.
\end{proof}
For the class $\mathcal{A}\left(0\right)$ we have a complete characterization in terms of the multiplicity function.
\begin{lem}\label{lem:shiftinv}
Given $\ell\in\bbN$, the exponential system $E\left(\ell\bbZ\right)$
is a Riesz sequence in $L^{2}\left(\bbS\right)$ if and only if 
\begin{equation}
\left|\left\{ t\in\bbT\,|\,\nu_{\ell,\bbS}\left(\nicefrac{t}{\ell}\right)=0\right\} \right|=0.\label{eq:mult0}
\end{equation}
\end{lem}
\begin{rem*}
	Lemma \ref{lem:shiftinv} can also be deduced from \cite{Chr}, Theorem 7.2.3. Below we give an alternative proof.
	
\end{rem*}
\begin{proof}
Suppose $\left|\left\{ \nu_{\ell,\bbS}\left(\nicefrac{t}{\ell}\right)=0\right\} \right|=0$
for some $\ell\in\bbN$. This implies that for almost every point
$t\in\left[0,\frac{2\pi}{\ell}\right)$ one can find $j\in\left\{ 0,\ldots,\ell-1\right\} $
such that 
\begin{equation}
t+\frac{2\pi j}{\ell}\in\bbS.\label{eq:mult01}
\end{equation}
Now, take any positive integer $N$ and consider $Q_{N}\left(\ell t\right)$
which is a $\frac{2\pi}{\ell}$-periodic function and so 
\[
1=\intop_{\bbT}\left|Q_{N}\left(\ell t\right)\right|^{2}dt=\ell\intop_{\left[0,\nicefrac{2\pi}{\ell}\right)}\left|Q_{N}\left(\ell t\right)\right|^{2}dt
\]
By (\ref{eq:mult01}) we have 
\[
\intop_{\bbS}\left|Q_{N}\left(\ell t\right)\right|^{2}dt\geq\intop_{\left[0,\nicefrac{2\pi}{\ell}\right)}\left|Q_{N}\left(\ell t\right)\right|^{2}dt=\frac{1}{\ell}
\]
and so $E\left(\ell\bbZ\right)$ is a Riesz sequence in $L^{2}\left(\bbS\right)$. 

For the other direction, suppose $\left|\left\{ \nu_{\ell,\bbS}\left(\nicefrac{t}{\ell}\right)=0\right\} \right|>0$
for some $\ell\in\bbN$. This means that there exists a set $E_{\ell}\subset\left[0,\frac{2\pi}{\ell}\right)$
of positive measure such that for all $j\in\left\{ 0,\ldots,\ell-1\right\} $
we have 
\[
E_{\ell}+\frac{2\pi j}{\ell}\in\bbT\backslash\bbS.
\]
Let $\eps>0$ and find $N$ such that the polynomial 
\[
P\left(\ell t\right)=\frac{1}{\sqrt{\ell\left|E_{\ell}\right|}}\sum_{\left|k\right|\leq N}\widehat{\mathds{1}_{E_{\ell}}}\left(k\right)e^{ik\ell t}
\]
satisfies 
\[
\intop_{\bbS}\left|P\left(\ell t\right)\right|^{2}dt<\eps\,,\,\intop_{\bbT}\left|P\left(\ell t\right)\right|^{2}dt>\frac{1}{2}.
\]
This is possible due to the fact that the $L^{2}\left(\bbT\right)$
function $\sum_{k\in\bbZ}\widehat{\mathds{1}_{E_{\ell}}}\left(k\right)e^{ik\ell t}$
vanishes at almost every point of $\bbS$. Since $\eps>0$ is arbitrary
we deduce that $E\left(\ell\bbZ\right)$ is not a Riesz sequence in
$L^{2}\left(\bbS\right)$.
\end{proof}

\end{document}